\documentclass[12pt,twoside]{amsart}
\usepackage{amssymb}
\usepackage{a4wide}
\usepackage[latin1]{inputenc}
\usepackage[T1]{fontenc}
\usepackage{times}

\def\hpq0{h^{p,q}_{\leq 0}}
\def\Hpq0{\H_{\leq 0}^{p,q}}

\def\R{{\mathbb R}}

\def\C{{\mathbb C}}

\def\F{{\mathcal F}}

\def\B{{\mathcal B}}

\def\D{\mathcal{D}}

\def\dof{\dot{\phi_t}}

\def\B{{\mathcal B}}
\def\H{{\mathcal H}}
\def\E{{\mathcal E}}

\def\Re{{\rm Re\,  }}

\def\be{\begin{equation}}
\def\ee{\end{equation}}

\newtheorem{thm}{Theorem}[section]
\newtheorem{lma}[thm]{Lemma}

\newtheorem{prop}[thm]{Proposition}

\theoremstyle{definition}

\theoremstyle{remark}

\newtheorem{preremark}{Remark}
\newtheorem{preex}{Example}

\numberwithin{equation}{section}

\title[]
{Symmetrization of plurisubharmonic and convex functions}

\address{Department of Mathematics\\Chalmers University
  of Technology \\
  and Department of Mathematics\\University of G\"oteborg\\S-412 96
  G\"OTEBORG\\SWEDEN,\\} 

\email{ robertb@chalmers.se;bob@chalmers.se}

\author[]{ Robert J Berman and Bo Berndtsson}

\begin{document}

\begin{abstract}
We show that Schwarz symmetrization does not increase the Monge-Ampere energy for $S^1$-invariant plurisubharmonic functions in the ball. As a result we derive a sharp Moser-Trudinger inequality for such functions. We also show that similar results do not hold for other  balanced domains except for complex ellipsoids and discuss related questions for convex functions.
\end{abstract}

\maketitle

\section{Introduction.}

If $\phi$ is a real valued function defined in a domain $\Omega$ in $\R^n$, its {\it Schwarz symmetrization}, see \cite{Kesavan}, is a radial function, $\hat\phi(x)=f(|x|)$, with $f$ increasing, that is equidistributed with $\phi$. The latter requirement means  that for any real $t$, the measure of the corresponding sublevel sets of $\phi$ and $\hat \phi$ are equal, i e
$$
|\{\phi<t\}|=|\{\hat\phi <t\}|=:\sigma(t).
$$
Notice that since $\hat\phi$ is radial, its natural domain of definition is a ball, $B$. Moreover, as $t$ goes to infinity, $\sigma(t)$ tends to the volume of $\Omega$ and also to the volume of $B$.  Thus the   volume of $B$  equals the volume of $\Omega$. 

Since $\phi$ and $\hat\phi$ are equidistributed, any integrals of the form 
$$
\int_\Omega F(\phi) dx
$$
and
$$
\int_B F(\hat\phi) dx,
$$
where $F$ is a measurable function of a real variable, are equal.

One fundamental property of symmetrization is that many other quantities measuring the 'size' of a function, decrease under symmetrization. The prime examples of this are energy integrals
$$
\int_\Omega |\nabla \phi|^p,
$$
for $p\geq 1$, see \cite{Kesavan}. By the Polya-Szeg\"o theorem
\be
\int_B |\nabla \hat\phi|^p\leq\int_\Omega |\nabla \phi|^p  ,
\ee
if $\phi$ vanishes on the boundary of $\Omega$.
This means that e g the study of Sobolev type inequalities
$$
(\int_\Omega |\phi|^q)^{1/q}\leq A(\int_\Omega |\nabla\phi|^p)^{1/p}
$$
is immediately reduced to the radial case, which is  a one-variable problem. 

Before we go on we remark that the inequality (1.1) is strongly related to the isoperimetric inequality. Indeed, if we take $p=1$ and  $\phi$ to be the characteristic function of $\Omega$, then as noted above, the corresponding ball has the same volume as $\Omega$. On the other hand, The $L^1$-norm of $\nabla\phi$ (taken in the sense of distributions), is the area of the boundary of $\Omega$. It follows that the area of the boundary of $\Omega$ is not smaller than the area of the sphere bounding the same volume, which is the isoperimetric inequality. The isoperimetric inequality is also the main ingredient in the proof of (1.1). 

In this  paper we will investigate analogs of (1.1) for another type of energy functional which is of interest in connection with convex and plurisubharmonic functions. In the case of convex functions, the functional is
$$
\E(\phi):=\int (-\phi) MA(\phi),
$$
where $MA(\phi)$ is the {\it Monge-Ampere measure}  of $\phi$, defined as
$$
MA(\phi):=\det(\phi_{j k}) dx
$$
when $\phi$ is twice differentiable. We will only consider this functional when $\phi$ vanishes on the boundary. In the one-dimensional case we can then integrate by parts, so that 
$$
\E(\phi)=\int |d\phi|^2
$$
is  the classical energy. In the general case we can also integrate by parts, and then find that $\E$ is still an $L^2$ -norm of  $d\phi$, but the norm of the differential is measured by the Hessian of $\phi$. This is why this functional makes sense primarily for convex functions. 

We also denote by $\phi$ the corresponding functional for plurisubharmonic functions. $\Omega$ is then a domain in $\C^n$ and we let
$$
\E(\phi):=\frac{1}{n+1}\int_\Omega (-\phi) (dd^c\phi)^n,
$$
called the {\it pluricomplex} or {\it Monge-Ampere} energy. (Notice that our normalization here is slightly different from the real case. It also differs from the definition used in \cite{Ber-Ber} by a sign; here we have chosen signs so that the energy is nonnegative.) It is  defined for plurisubharmonic functions, vanishing on the boundary, satisfying some extra condition so that the complex Monge-Ampere measure is well defined. Just like in the real case, the pluricomplex energy equals the classical (logarithmic) energy when the complex dimension is one.

We start with the case of plurisubharmonic functions. The first problem is that the Schwarz symmetrization of a plurisubharmonic function is not necessarily plurisubharmonic, so the Aubin-Yau energy is not naturally defined.  Indeed, already when the complex dimension is one and we take
$$
\phi(z)=\log|(z-a)/(1-\bar a z)|
$$
to be the Green kernel, $\hat\phi$ is subharmonic only if $a=0$, so that $\phi$ is already radial. (We thank Joaquim Ortega and Pascal Thomas for providing us with this simple example.) Our first observation is that if we consider only functions (and domains) that are {\it $S^1$-invariant}, i e invariant under the map $z\mapsto e^{i\theta}z$, then the symmetrization, $\hat\phi$ , of a plurisubharmonic function, $\phi$, is again plurisubharmonic. Thus it is meaningful to consider its energy and the main result we prove is that 
\be
\E(\hat\phi)\leq\E(\phi)
\ee
when $\Omega$ is a ball. The condition of $S^1$-symmetry of course makes this result trivial when $n=1$, but notice that it is a rather weak restriction in high dimensions as it only means invariance under a one dimensional group.

In section 4, we  study the corresponding problems for convex functions. In that case  convexity is preserved under Schwarz symmetrization (this must be well known but we include a proof in section 4), so we need no extra condition (like $S^1$-invariance). We then show that  for convex functions in the ball, vanishing on the boundary, symmetrization decreases the Monge-Ampere energy, just like in the complex case, and following a similar argument.

It is natural to ask if these symmetrization results also hold for other domains than the ball. In the classical case of the Polya-Szeg\"o theorem one symmetrizes the domain and the function at the same time and it is the symmetrization of the domain that is most clearly linked to the isoperimetric inequality. It turns out that the counterpart to this for Monge-Ampere energy does not hold. Indeed, in section 2 we prove the somewhat surprising fact that our symmetrization result in the complex case holds if and only if the domain  $\Omega$ is an ellipsoid, i e the image of the Euclidean ball under a complex linear transformation. The proof is based on the interpretation of  $S^1$-invariant domains as unit disk bundles of line bundles over projective space, and the proof uses the Bando-Mabuchi uniqueness theorem for K\"ahler-Einstein metrics on $\mathbb{P}^n$. 

In the real case the situation is a little bit more complicated. It was first shown by Tso, \cite{Tso}, that the symmetrization inequality fails in general: There is a convex domain and a convex function vanishing on the boundary of that domain, whose Schwarz symmetrization has larger energy. In section 4 we first give a general form of Tso's example and relate it to Santal\`o's inequality.  
We show that if the symmetrization inequality holds for a certain domain $\Omega$, then $\Omega$ must be a maximizer for the  {\it Mahler volume}, i e for the product of the volume of $\Omega$ with the volume of its polar body, $\Omega^{\circ}$. By (the converse to) Santal\`o's inequality this means that $\Omega$ is an ellipsoid. Thus we arrive at the same conclusion as in the complex case, but this time for a completely different reason, that seems to have no counterpart in the complex setting. We then argue that if we redefine the Monge-Ampere energy in the real setting by dividing by the Mahler volume we get an energy functional that behaves more like in the complex setting, and for which the phenomenon discovered by Tso disappears. That the symmetrization inequality holds for this renormalized energy is thus a weaker statement. Nevertheless we show, by an argument similar to the one used in the complex case, that even the weaker inequality  holds only for ellipsoids.

The origin of this paper is our previous article \cite{Ber-Ber}, where we studied Moser-Trudinger inequalities of the form
\be
\log\int_\Omega e^{-\phi}\leq A \E(\phi) +B,
\ee
for plurisubharmonic functions in $\Omega$ that vanish on the boundary. It follows immediately from (1.2) that, when $\Omega$ is a ball and $\phi$ is $S^1$-invariant, the proof of inequalities of this type can be reduced to the case of radial functions. In \cite{Ber-Ber} we proved a Moser-Trudinger inequality using geodesics in the space of plurisubharmonic functions. Here we will use instead symmetrization, but we point out that the proof of our main result (1.2) also uses geodesics.  By  the classical results of Moser, \cite{Moser}, we then obtain in section 3  a sharpening of the Moser-Trudinger inequality from \cite{Ber-Ber}. Symmetrization was the main tool used by Moser to study the real variable Moser-Trudinger inequality, and it is interesting to note that symmetrization applied to (1.3) leads to the same one variable inequality as in Moser's case. As a result we deduce that if $\phi$ is $S^1$-invariant and has finite energy that we can normalize to be equal to one , then
$$
\int_B e^{n(-\phi)^{(n+1)/n}} < \infty.
$$
We do not know if this estimate holds without the assumption of $S^1$-symmetry.

\section{Symmetrization of plurisubharmonic functions}

The proofs in this section are based on a result from \cite{Berg} that we first recall. We consider a pseudoconvex domain $\D$ in $\C^{n+1}$ and its $n$-dimensional slices 
$$
D_t=\{z\in \C^n; (t,z)\in \D\}
$$
where $t$ ranges over (a domain in ) $\C$. We say that a domain $D$ in $\C^n$ is $S^1$-invariant if $D$ is invariant under the map
$$
z\mapsto e^{i\theta} z=( e^{i\theta}z_1, ...e^{i\theta}z_n)
$$
for all $\theta$ in $\R$. A function (defined in a $S^1$-invariant domain) is $S^1$-invariant if $f(e^{i\theta} z)=f(z)$ for all real $\theta$.

\begin{thm} Assume that $\D$ is a pseudoconvex domain in $\C^{n+1}$ such that all its slices $D_t$ are connected and $S^1$-invariant. Assume also that the origin belongs to $D_t$ when $t$ lies in a domain $U$ in $\C$. Then
$$
\log |D_t| 
$$
is a superharmonic function of $t$ in $U$.
\end{thm}
Theorem 2.1 is a consequence of the main result in \cite{Berg} which says that if $B_t(z,z)$ is the Bergman kernel on the diagonal for the domain $D_t$, then $\log B_t(z,z)$ is plurisubharmonic in $\D$. The hypotheses on $D_t$ in the theorem imply that 
$$
B_t(0,0)=|D_t|^{-1}
$$
which gives the theorem. Theorem 2.1 can be seen as a complex variant of the (multiplicative form of) the {\it Brunn-Minkowski inequality}, which says that if $\D$ is instead convex in $\R^{n+1}$, and the $n$-dimensional slices are defined in the same way then 
$$
\log |D_t|
$$
is a concave function of $t$, without any extra assumptions on the slices. The Brunn-Minkowski inequality  replaces the use of Theorem 2.1 when we later consider energies of convex functions. Then we will use the stronger fact that even $\sigma^{1/n}$ is concave in the real setting. 

In the proofs below we will use the following lemma on symmetrizations of subharmonic functions.
\begin{lma} Let $u$ be a smooth subharmonic function defined in an open set $U$  in $\R^N$, and assume that $u$ vanishes on the boundary of $U$. Let
$$
\sigma(t):=\{x; u(x)<t\}|
$$
for $t<0$. Then $\sigma$ is  strictly increasing on the interval $(\min u, 0)$ and the Schwarz symmetrization of $u$, $\hat u$, equals $ g(|x|)$ where 
$$
g(r)=\sigma^{-1}(c_N r^N)
$$
where $c_N$ is the volume of the unit ball in $\R^N$, when 
$$
c_N r^N>|U_{\min(u)}|.
$$
 When   
$$
c_N r^N\leq|U_{\min(u)}|,
$$ 
$g(r)=\min(u)$
\end{lma}
\begin{proof} Denote by  $U_t$ the domain  where $u<t$. Assume  that $|U_t|=|U_{t+\epsilon}|$ for some $\epsilon>0$. By Sard's lemma, some $s$ between $t$ and $t+\epsilon$ is a regular value of $u$, so the boundary of $U_s$ is smooth. By the Hopf lemma, the gradient of $u$ does not vanish on the boundary of $U_s$ unless $u$ is constant in $U_s$. In the latter case $s\leq\min u$. If this is not the case, i e if $s>\min u$, the coarea formula gives 
$$
\sigma'(s)=\int_{\partial U_s} dS/|\nabla u|>0
$$
which contradicts that $\sigma$ is constant  on $(t, t+\epsilon)$. This proves the first part of the lemma. The second part follows since
$$
\sigma((g(r))=|\{ g(|x|)<g(r)\}|= c_N r^N.
$$
\end{proof}

We can now easily prove the next basic result. We say that a domain $\Omega$ in $\C^n$ is {\it balanced} if for any $\lambda$ in $\C$ with $|\lambda|\leq 1$ and $z$ in $\Omega$, $\lambda z$ also lies in $\Omega$.
\begin{thm} Let $\Omega$ be a balanced domain in $\C^n$. Let $\phi$ be an $S^1$-invariant plurisubharmonic function in $\Omega$. Then $\hat\phi$, the Schwarz symmetrization of $\phi$ is plurisubharmonic.
\end{thm}

\begin{proof} We may of course assume that $\phi$ is smooth so that the previous lemma applies. By definition, $\hat\phi$ can be written
$$
\hat\phi(z)=f(\log|z|),
$$
so what we need to prove is that $f$ is convex. Since $\phi$ and $\hat\phi$ are equidistributed, for any real $t$,
$$
\sigma(t):=|\{z\in\Omega; \phi(z)<t\}|=|\{z\in\Omega; \hat\phi(z)<t\}=|\{z; |z|<\exp(f^{-1}(t))\}|.
$$
Hence
$$
f^{-1}(t)=n^{-1}\log\sigma(t) +b_n.
$$
Since $\sigma$ is increasing, $f^{-1} $ is also increasing. Therefore $f$ is convex precisely when $f^{-1}$ is concave, i e when $\log\sigma$ is concave. 

Consider the domain in $\C^{n+1}$
$$
\D=\{ (\tau,z); z\in \Omega \quad \text{and} \quad \phi(z)-\Re \tau <0\}.
$$
Then, if $t=\Re \tau$, $\sigma(t)=|D_\tau|$. Note that $\D$ is pseudoconvex since $\phi-\Re\tau$ is plurisubharmonic and we claim that $\D$ also satisfies all the other conditions of Theorem 2.1 . 

Let $z$ lie in $D_\tau$ for some $\tau$. The function $\gamma(\lambda):=\phi(\lambda z)$ is then subharmonic in the unit disk, and moreover it is radial, i e
$$
\gamma(\lambda)=g(|\lambda|),
$$
where $g$ is increasing. Therefore the whole disk $\{\lambda z\}$ is contained in $D_\tau$. In particular the origin lies in any $D_\tau$, and the origin can be connected with $z$ by a curve, so ${\D}_{\tau}$ is connected. Thus Theorem 2.1 can be applied and we conclude that 
$$
\log\sigma(\Re\tau)=\log |D_\tau|
$$
is a superharmonic function of $\tau$. Since this function only depends on $\Re \tau$ it is actually concave, and the proof is complete.
\end{proof}

The next theorem is the main result of this paper, and here we need to assume that $\Omega$ is a ball. See the remarks below  for a discussion of the problem with considering more general domains.
\begin{thm} Let $\phi$ be plurisubharmonic in the unit ball, and assume that $\phi$ extends continuously to the closed ball with zero boundary values. Assume also that $\phi$ is $S^1$-invariant, and let $\hat\phi$ be the Schwarz symmetrization of $\phi$. Then
$$
\E(\hat\phi)\leq\E(\phi).
$$
\end{thm}

In the proof of Theorem 2.1 we used the geometrically obvious fact that the inverse of an increasing concave function is convex. We will need a generalization of this that we state as a lemma.
\begin{lma}
Let $a(s,t)$ be a concave function of two real variables. Assume
$a$ is strictly increasing with respect to $t$, and let $t=k(s,x)$
be the inverse of $a$ with respect to the second variable for $s$
fixed, so that $a(s,k(s,x))=x$. Then $k$ is convex as a function
of both variables $s$ and $x$. \end{lma}
\begin{proof}
Assume not. After choosing a new origin, there is then a point $p=(s_{0},x_{0})$
such that \[
k(0,0)>(k(p)+k(-p))/2.\]
 Since $a$ is strictly increasing with respect to $t$ \[
0=a(0,k(0,0))>a((s_{0}-s_{0})/2,(k(p)+k(-p))/2)\geq\]
 \[
\geq(a(s_{0},k(p))+a(-s_{0},k(-p)))/2=(x_{0}-x_{0})/2=0.\]
 This is a contradiction . 
\end{proof}

In the sequel  we shall use well known facts about geodesics and subgeodesics in the space of plurisubharmonic functions on the ball, see \cite{Ber-Ber} for proofs. These are curves
$$
\phi_t(z)=\phi(t,z)
$$
where $t$ is real parameter, here varying between 0 and 1. By definition, $\phi_t$ is a subgeodesic if $\phi(\Re \tau,z)$ is plurisubharmonic as a function of $(\tau,z)$, and it is a geodesic if moreover this plurisubharmonic function solves the homogenuous complex Monge-Ampere equation
$$
(dd^c\phi)^{n+1}=0.
$$
We  also assume throughout that $\phi_t$ vanishes for $|z|=1$. It is not hard to see that if $\phi_0$ and $\phi_1$ are two continuous plurisubharmonic functions in the ball, vanishing on the boundary, then they can be connected with a bounded geodesic, see \cite{Ber-Ber}. Here we shall first assume that $\phi_0$ and $\phi_1$ are smooth and can be connected by a geodesic of class $C^1$, and then get the inequality for general $\phi$ by approximation. 

We will  use the following three facts, for which we refer to \cite{Ber-Ber}. First, $\E(\phi_t)$ is an {\it affine} function of $t$ along any bounded geodesic. Second, $\E(\phi_t)$  is {\it concave} along a bounded subgeodesic. On the other hand, if $\phi_0$ and $\phi_1$ are plurisubharmonic and vanish on the boundary and we let $\phi_t=t\phi_1 +(1-t)\phi_0$ for $t$ between 0 and 1, then $\E(\phi_t)$ is {\it convex}. Finally, if $\phi_t$ is of class $C^1$, then $\E(\phi_t)$ is differentiable  with derivative
\be
\frac{d}{dt}\E(\phi_t)=\int_B-\dof (dd^c_z\phi_t)^n.
\ee

In the proof of Theorem 2.3 we fix a plurisubharmonic $\phi$ in the ball, that we assume smooth. We put $\phi=\phi_1$ and connect it with $\phi_0$, chosen to satisfy an equation 
$$
(dd^c\phi_0)^n= F(\phi_0),
$$
where $F$ is some smooth function of a real variable. Actually, it is not hard to check that any smooth increasing radial function satisfies such an equation.
 We also first assume that $\phi=\phi_1$ and $\phi_0$ can be connected with a $C^1$ geodesic $\phi_t$. We then take the Schwarz symmetrization of each $\phi_t$ and obtain another curve $\hat\phi_t$. The next proposition shows that the new curve is a subgeodesic. 
\begin{prop} Let $\phi_{t}$ be a subgeodesic of $S^{1}$-invariant
plurisubharmonic functions. Then $\hat{\phi}_{t}$ is also a subgeodesic. 
\end{prop}
\begin{proof}
Let $\phi_{s}$ be a subgeodesic which we may assume to be smooth. 
Let 
$$
A(s,t)=|\{z,\phi_{s}(z)<t\}|.
$$
 It follows again from Theorem 2.1 
that $a:=\log A$ is a concave function of $s$ and $t$ together.
As in the proof of Theorem 2.3  all we need
to prove is that the inverse of $a$ with respect to $t$ (for $s$
fixed), $k(s,x)$) is convex with respect to $s$ and $t$ jointly.
But this is precisely the content of  Lemma 2.4.
\end{proof}
We now first sketch the principle of the argument and fill in some details and change the set up a little bit  afterwords.
Consider the energy functionals along the two curves $\phi_t$ and $\hat\phi_t$,  $\E(\phi_t)=:g(t)$ and $\E(\hat\phi_t)=h(t)$. Since $\phi_0$ is already radial, $g(0)=h(0)$, and we want to prove that $g(1)\geq h(1)$. We know that $g$ is affine and that $h$ is concave, so this follows if we can prove that $g'(0)=h'(0)$. 
But 
$$
g'(0)=\int-\dot\phi_0 (dd^c\phi_0)^n,
$$
since the geodesic is $C^1$.
We also claim  that we can arrange things so that 
$$
h'(0)=\int -\frac{d\hat\phi_t }{dt}|_{t=0}\,(dd^c\phi_0)^n.
$$
By the choice of $\phi_0$,
$$
g'(0)=\int-\dot\phi_0 F(\phi_0)=\frac{d}{dt}|_{t=0}\int -G(\phi_t),
$$
if $G'=F$. Similarily
$$
h'(0)=\frac{d}{dt}|_{t=0}\int -G(\hat\phi_t).
$$
But, since $\phi_t$ and $\hat\phi_t$ are equidistributed
$$
\int -G(\phi_t)=\int -G(\hat\phi_t)
$$
for all $t$. 
Hence $g'(0)=h'(0)$ and the proof is complete. 

\bigskip

It remains to see why we can assume that the geodesic $\phi_t$ is $C^1$ and to motivate the claim about the derivative of $h$. First, since we have assumed that $\phi_0$ and $\phi_1$ are smooth up to the boundary, we can by a max construction assume that they are booth equal to $A \log((1+|z|^2)/2)$ for some large $A>0$, when $|z|>(1-\epsilon)$. Then $\phi_0$ and $\phi_1$ can be extended to psh functions in all of $\C^n$, equal to  $A \log((1+|z|^2)/2)$ outside of the unit ball. In fact, we can even consider them as metrics on a line bundle $\mathcal{O}(A)$ over $\mathbb{P}^n$. It the follows from Chen's theorem, \cite{Chen},  that they can be connected by a $C^1$ geodesic in the space of metrics on $\mathcal{O}(A)$. It is easy to see that this geodesic must in  fact be equal to  $A \log((1+|z|^2)/2)$ for $|z|> 1-\epsilon$ for some positive $\epsilon$. In particular, it vanishes on the boundary of the ball, and $\dot\phi_t$ is identically zero near the boundary. 

To handle the claim about the derivative of $h$ we change the setup a little bit. $\hat\phi_0=\phi_0$ is smooth and we can approximate $\hat\phi_1$ from above by a smooth radial plurisubharmonic  function. Now connect these two smooth functions by a geodesic, $\psi_t$, which can be take to be $C^{(1,1)}$ by the above argument. (As a matter of fact it will even be smooth, since geodesics between radial functions come from geodesics between smooth convex functions, which are smooth). Let
$$
\E(\psi_t)=:k(t).
$$
Since $\psi_t\geq\hat\phi_t$, $-\dot\psi_0\leq -\dot{\hat\phi}_0$. We then apply the above argument to $k$ instead of $h$ and find that $k(1)\leq g(1)$. Taking limits as $\psi_1$ tends to $\hat\phi_1$ we  conclude the proof.

\subsection{Other domains.}

Let us consider a smoothly bounded balanced domain $\Omega$ in $\C^n$ which we can write as
$$
\Omega=\{z; u_\Omega(z)<0\}
$$
where $u_\Omega$ is {\it logarithmically homogenuous} , i e $u_\Omega(\lambda z)=\log |\lambda| +u_\Omega(z)$. Indeed,  $u_\Omega$ is the logarithm of the Minkowski functional for $\Omega$. We first claim that if $\Omega$ is pseudoconvex then $u_\Omega$ is plurisubharmonic.
\begin{lma} Let $u$ be a smooth function such that 
$$
D:=\{(w,z); u(z)-\Re w<0\}
$$
is pseudoconvex. Then $u$ is plurisubharmonic.
\end{lma}
\begin{proof}  At a point $z$ where $du=0$ the Levi form of the boundary of $D$ is precisely $dd^c u$ so if $D$ is pseudoconvex $dd^c u\geq 0$ at such points. The general case is reduced to this by subtracting a linear form $\Re a\cdot z$ from $u$ and considering the biholomorphic transformation $(w,z)\mapsto (w+a\cdot z, z)$.
\end{proof}

Since the set $\{u_\Omega(z)-\Re w<0\}=\{u_\Omega(z e^{-w})<0\}$ is pseudoconvex it follows from the lemma that $u_\Omega$ is plurisubharmonic. Let us now consider $S^1$-invariant functions in $\Omega$ of the form $\phi(z)=f(u_\Omega)$ where $f$ is convex. If we normalize so that the volume of $\Omega$ equals the volume of the unit ball it is clear that the Schwarz symmetrization of $\phi$, $\hat\phi$ is $f(\log|z|)$.
\begin{prop}
If $\phi=f(u_\Omega)$ with $f$ convex, the Monge-Ampere energy of $\phi$ equals
$$
\int_\Omega (-\phi)(dd^c\phi)^n=2^{-n}\int_{-\infty}^0 (f')^{n+1}(t)dt.
$$
In particular, the energy of $\phi$ is equal to the energy of $\hat\phi$, the Schwarz symmetrization of $\phi$.
\end{prop}
In the proof we use  the next lemma.
\begin{lma}If $\phi=f(u_\Omega)$ with $f$ convex
$$
\int_{u_\Omega<s} (dd^c\phi)^{n}= 2^{-n}f'(s)^n.
$$
\end{lma}
\begin{proof}
$$
\int_{u_\Omega<s} (dd^c\phi)^{n}=\int_{u_\Omega=s}d^c\phi\wedge(dd^c\phi)^{n-1}=
f'(s)^n\int_{u_\Omega=s}d^cu_\Omega\wedge(dd^c u_\Omega)^{n-1}.
$$
But
$$
\int_{u_\Omega=s}d^cu_\Omega\wedge(dd^c u_\Omega)^{n-1}=\int_{u_\Omega<s}(dd^c u_\Omega)^n.
$$
Since $u_\Omega$ is log homogenous it satsifies the homogenous Monge-Ampere equation outside of the origin, so $(dd^c u_\Omega)^n$ is a Dirac mass at the origin. But $u_\Omega-\log|z|$ is bounded near the origin so this point mass must be same as 
$$
(dd^c\log|z|)^n= 2^{-n}.
$$
\end{proof}

\noindent {\it Proof of Proposition 2.8} First assume that $f(s)$ is constant for $s$ sufficiently large negative. Let
$$
\sigma(s):= \int_{u_\Omega<s} (dd^c\phi)^n.
$$
Then
$$
\E(\phi)=\int^0 -f(s)d\sigma(s)=\int^0 \sigma(s)f'(s)ds,
$$
so the formula for the energy folows from the previous lemma. 
The general case, when $f$ is not constant near $-\infty$  follows from approximation.

The last statement, that $\E(\phi)=\E(\hat\phi)$ then follows if $|\Omega|$ equals the volume of the unit ball, since then $\hat\phi=f(\log|z|)$. But then the same thing must hold in general since the energy is invariant under scalings.
\qed

\bigskip

Let us now define the '$\Omega$-symmetrization', $S_\Omega(\phi)$,  of a plurisubharmonic function $\phi$ in $\Omega$, vanishing on the boundary, $\phi$, as the unique function of the form $f(u_\Omega)$ that is equidistributed with $\phi$. Notice that if  the $\Omega$-symmetrization of $\phi$ equals $f(u_\Omega)$, then the Schwarz symmetrization is given by $\hat\phi= f(\log(|z|/R)$, if $R$ is chosen so that the volume of $\Omega$ equals the volume of the ball of radius $R$. The last part of proposition 2.8 then says that
$$
\E_\Omega(S_\Omega(\phi))=\E_B(\hat\phi),
$$
where we have put subscripts on $\E$ to emphasize over which domain we compute the energy, and $B$ denotes a ball of the same volume as $\Omega$. 
Notice also that it follows from Theorem 2.3 that $S_\Omega(\phi)$ is plurisubharmonic if $\phi$ is plurisubharmonic. Indeed, Theorem 2.3 says that $\hat\phi=f(\log|z|)$ is plurisubharmonic, i e that $f$ is convex and increasing, from which it follows that $f(u_\Omega)$ is plurisubharmonic.  
We therefore  see that to prove that the Schwarz symmetrization of
a function $\phi$ on $\Omega$ has smaller Monge-Ampere energy than $\phi$ is equivalent to proving that the $\Omega$-symmetrization of $\phi$ has smaller energy than $\phi$.

One might try to prove this by following the same method as in the proof of Theorem 2.4. The point where the proof breaks down is that we need to choose a reference function on $\Omega$ that satisfies an equation of the form
\be
(dd^c\phi_0)^n=F(\phi_0)
\ee
where $\phi_0$ is of the form $\phi_0=f(u_\Omega)$. This is easy if $\Omega$ is the ball so that $u_\Omega=\log|z|$ since $(dd^c\phi_0)^n$ then is invariant under the unitary group if $\phi_0$ is, and hence must be radial. Nothing of the sort holds for other domains. Since, outside the origin, 
$$
(dd^c\phi_0)^n=f''(u_\Omega)f'(u_\Omega)^{n-1}du_\Omega\wedge d^cu_\Omega\wedge (dd^c u_\Omega)^{n-1},
$$
what we  want is that the determinant of the Leviform
\be
du_\Omega\wedge d^cu_\Omega\wedge (dd^c u_\Omega)^{n-1}
\ee
be constant on all level surfaces of $u_\Omega$.
This is clearly true if $\Omega$ is a ball, and therefore also true if $\Omega$ is the image of a ball under a complex linear transformation. We shall next see that these are the only cases when this holds.
\begin{prop} Let $\Omega$ be a balanced domain in $\mathbb C^n$ and let $u_\Omega$ be the uniquely determined logarithmically homogenous (plurisubharmonic)  function that vanishes on the boundary of $\Omega$. Assume $u_\Omega$ satisfies the condition that  (2.3) be constant  on some, and therefore every , level surface of $u_\Omega$. Then $\Omega$ is an ellipsoid
$$
\Omega=\{z; \sum a_{j k} z_j\bar z_k <1\}
$$
for some positively definite matrix $A=(a_{j k})$.
\end{prop}

\begin{proof}
We will use the relation between logarithmically homogenous functions on $\mathbb C^n$ and metrics on the tautological line bundle $\mathcal{O}(-1)$ on $\mathbb P^{n-1}$. Recall that $\mathbb P^{n-1}$ is the quotient of $\mathbb C^n\setminus\{0\}$ under the  equivalence relation $z\sim \lambda z$ if $\lambda$ is a nonzero complex number. Let $p(z)=[z]$ be the projection map from $\mathbb C^n\setminus\{0\}$ to $\mathbb P^{n-1}$; $[z]$ being the representation of a point in homogenous coordinates. Then $\mathbb C^n\setminus\{0\}$ can be interpreted as the total space of $\mathcal{O}(-1)$, minus its zero section. A logarithically homogenous function like $u_\Omega$ can then be written $u_\Omega=\log |z|_h$ for some metric $h$ on $\mathcal{O}(-1)$. In an affine chart $[z]=[(1,\zeta)]$ on $\mathbb P^{n-1}$ with associated trivialization of $\mathcal{O}(-1)$ ; $z=(\lambda, \lambda \zeta)$,
$$
|z|^2_h=|\lambda|^2 e^{\psi(\zeta)}
$$
where $-\psi$ is a local representative for the metric $h$ on $\mathcal{O}(-1)$. Hence
\be
u_\Omega=\log |\lambda| +(1/2)\psi(\zeta).
\ee
\bigskip

Let us now look at the form (2.3). If it is constant on level surfaces it must be equal to $ F(u) dz\wedge d\bar z$ for some function $F(u)$. By log-homogenuity, the form is moreover homogenous of degree $-2n$, so we must have
$$
F(u)= Ce^{-2n u}. 
$$
Changing coordinates to $(\lambda, \lambda\zeta)$ we get
$$
du_\Omega\wedge d^cu_\Omega\wedge (dd^c u_\Omega)^{n-1}= C e^{-2n u} |\lambda|^{2n-2} d\lambda\wedge d\bar\lambda\wedge d\zeta\wedge d\bar \zeta.
$$
On the other hand, by (2.4),
$$
du_\Omega\wedge d^cu_\Omega\wedge (dd^c u_\Omega)^{n-1}=C' |\lambda|^{-2} d\lambda\wedge d\bar\lambda\wedge (dd^c\psi)_{n-1}.
$$
Hence
$$
(dd^c\psi)_{n-1}= C'' e^{-n\psi} d\zeta\wedge d\bar \zeta.
$$
This means precisely that the metric $n\psi$ on the anticanonical line bundle $\mathcal{O}(n)$ on $\mathbb P^{n-1}$ solves the K\"ahler-Einstein equation . But all such metrics can be written (on the total space)
$$
\log |\chi|_h^2= n\log |A z|^2
$$
where $z$ are the standard coordinates on $\mathbb C^n$ and $A$ is a positively definite matrix. (This follows \newline  e g from the Bando-Mabuchi uniqueness theorem, which says that any K\"ahler-Einstein metric can be obtained from the standard metric $\log |z|^2$ via a holomorphic automorphism.) Hence
$$
u_\Omega=\log |Az|,
$$
so $\Omega$ is an ellipsoid.

\end{proof}

We shall now finally show that the relevance of the form (2.3)  is not just an artefact of the proof. Indeed, we shall show that if the symmetrization inequality
$$
\E_B(\hat\phi)\leq\E_\Omega(\phi)
$$
holds for all $S^1$-invariant plurisubharmonic functions $\phi$ in $\Omega$ that vanish on the boundary, then $\Omega$ must satisfy the hypothesis of proposition 2.10, and therefore be an ellipsoid.

 Let $\psi_0=f_0(\log|z|)$ be a function in the ball that solves a Kahler-Einstein type equation
$$
(dd^c\psi)^n= ce^{-\psi} i^{n^2}dz\wedge d\bar z.
$$
Then $\psi_0$  is  a critical point  for a functional of the type
$$
\F_B(\phi):=\log \int_B e^{-\phi}-c'\E_B(\phi) ,
$$
see \cite{Ber-Ber} for more on this. 
From this it follows that
\be
\F_B(\psi')\leq\F_B(\psi_0)
\ee
for all $S^1$-invariant plurisubharmonic functions in the ball that vanish on the boundary. This is explained in  \cite{Ber-Ber}, so we just indicate the argument here. The point is that the functional $\F_B$ is concave along geodesics in the space of $S^1$-invariant plurisubharmonic functions that vanish on the boundary. This follows from two facts. First, the energy term is affine along geodesics, and second, the function
$$
\log \int_B e^{-\phi_t}
$$
is concave under (sub)geodesics. The latter fact follows again from the main result in \cite{Berg} on plurisubharmonic variation of Bergman kernels, since
$$
( \int_B e^{-\phi_t})^{-1}
$$
is the Bergman kernel at the origin for the weight $\phi_t$ if $\phi_t$ is $S^1$-invariant. Given the concavity of $\F_B$ it then follows that a critical point is a maximum, i e (2.5) holds.

\bigskip

Let us now consider the analogous functional defined on functions on $\Omega$,
$$
 \F_\Omega(\phi):=\log \int_\Omega e^{-\phi}-c'\E_\Omega(\phi) .
$$
Assume, to get a contradiction that $\E_\Omega(S_\Omega(\phi)\leq \E_\Omega(\phi)$. Then $\F_\Omega$ increases under $\Omega$-symmetrization. Moreover, by proposition 2.8, $\F_\Omega(S_\Omega(\phi))=\F_B(\hat\phi)$ so 
$$
 \F_\Omega(\phi)\leq \F_B(\hat\phi)\leq \F_B(\psi_0),
$$
where $\psi_0$ is the K\"ahler-Einstein potential discussed above. Hence the maximum of the left hand side over all $\phi$ is attained for 
$$
\phi=f_0(u_\Omega).
$$

But then it is easy to see that $\phi$ solves  the same Kahler-Einstein equation as $\psi_0$. Indeed, at least if $\Omega$ is strictly pseudoconvex, $\phi$ is strictly plurisubharmonic ouside the origin. Therefore small perturbations of $\phi$ are still plurisubharmonic and  the variational equation for $\F_\Omega$ is just the K\"ahler-Einstein equation.  In particular, $\phi$ solves an equation of type (2.2) in $\Omega$, which we have seen is possible only if $\Omega$ is an ellipsoid. 
We summarize the discussion in the next theorem.
\begin{thm}
Let $\Omega$ be a strictly pseudoconvex balanced domain for which the symmetrization inequality
$$
\E_B(\hat\phi)\leq\E_\Omega(\phi)
$$
holds for all $S^1$-invariant plurisubharmonic $\phi$ that vanish on the boundary. Then $\Omega$ is an ellipsoid.
\end{thm}

\section{A sharp Moser-Trudinger inequality for $S^1$-invariant functions.}
Our results in the previous section, together with Moser's inequality imply rather easily the next estimate.

\begin{thm} Let $\phi$ be a smooth $S^1$-invariant plurisubharmonic function in the unit ball
that vanishes on the boundary. Let 
$$
\E:=\E(\phi).
$$
Then
$$
\int_\B e^{n\E^{-{1/n}} (-\phi)^{n+1)/n}}\leq C
$$
where $C$ is an absolute constant.
\end{thm}

In the proof we may by our main result on symmetrization assume that $\phi(z)=
f(\log|z|)$ is a radial function. The main result of Moser, \cite{Moser}, is that if
$w$ is an increasing function on $(-\infty, 0)$ that vanishes when $t$ goes to zero and
satisfies
$$
\int_{-\infty}^0 (-w')^{n+1} dt\leq 1
$$
then 
$$
\int_{-\infty}^0 e^{(-w)^{(n+1)/n}} e^{t} dt \leq C,
$$
where $C$ is an absolute constant. Applying this to $w_\kappa(s):= 
\kappa^{n/(n+1)}w(s/\kappa)$ we obtain that
$$
\int_{-\infty}^0 e^{\kappa (-w)^{(n+1)/n}} e^{\kappa t} dt \leq C/\kappa,
$$
under the same hypothesis. Next we have the following lemma.
\begin{lma} Let $f$ be an increasing convex function on $(-\infty, 0]$ with $f(0)=0$, and let $\phi(z)=f(\log |z|)$. Let  $F$  be a nonnegative measurable  function of one real variable. Then

\noindent
(a) 
$$
\int_\B F\circ \phi =a_n\int_{-\infty}^0 F\circ f e^{2n t} dt
$$
(with $a_n$ being the area of the unit sphere in $\C^n$).

\bigskip

 \noindent and

\noindent (b)
$$
\E = 2^{-n}\int_{-\infty}^0 (f')^{n+1} dt.
$$
\end{lma}
\begin{proof} The first formula follows from
$$
\int_BF(\phi)=\int^0 F\circ f d\sigma(t)
$$
where $\sigma=|\{z; |z|\leq e^t\}= \pi^n/n! e^{2nt}$. The second formula is a special case of Proposition 2.8. 

\end{proof}

Applying the scaled version of Moser's result with 
$$
-w= f E^{-1/(n+1)}2^{-n/(n+1)}
$$
and 
$\kappa =2n$ the theorem follows. 

\bigskip

To relate this to Moser-Trudinger inequalities of the form studied in \cite{Ber-Ber} we start with the elementary inequality for positive numbers $x$ and $\xi$
$$
x\xi\leq \frac{1}{n+1} x^{n+1} +\frac{n}{n+1}\xi^{(n+1)/n}
$$
(valid since $(n+1)$ and $(n+1)/n$ are dual exponents). This implies
$$
\xi\leq  \frac{1}{n+1} x^{n+1} +\frac{n}{n+1}\xi^{(n+1)/n}/x^{(n+1)/n}.
$$
Choose $x$ so that 
$$
x^{n+1}=E/(n+1)^n
$$
and take $\xi =(-\phi)$. Then
$$
-\phi\leq \frac{1}{(n+1)^{n+1}}E +n E^{-1/n}(-\phi)^{(n+1)/n}.
$$
Therefore Theorem 3.1 implies the sharp Moser-Trudinger inequality for $S^1$-invariant functions from \cite{Ber-Ber}
$$
\log\int e^{-\phi}\leq \frac{1}{(n+1)^{n+1}}\E(\phi) + B,
$$
with $B=\log C$ , $C$ the universal constant in Moser's estimate. 
\section{Symmetrization of convex functions}
First we note the following analog of Theorem 2.2.
\begin{thm} Let $\phi$ be a convex function defined in a convex domain $\Omega$ in $\R^n$ and let $\hat\phi$ be its Schwarz symmetrization. Then $\hat\phi$ is also convex.
\end{thm}
This fact should be well known but we include a proof in order to  emphazise the similarity with Theorem 2.2. By definition $\hat\phi(x)=g(|x|)$ for some increasing function $g$ and we need to prove that $g$ is convex (notice the change in convention as compared with the complex case where we wrote $\hat\phi(z)=f(\log |z|)$). As before
$$
\sigma(t):=|\{x\in\Omega; \phi(x)<t\}|= a_n (g^{-1}(t))^n,
$$
so it suffices to prove that $\sigma^{1/n}$ is {\it concave}. But if we put
$$
\D:=\{ (t,x); x\in\Omega \, \text{and}\, \phi(x)-t<0\},
$$
$\sigma(t)$ is the volume of the slices $D_t$. By the Brunn-Minkowski theorem (see section 2) it follows that $\sigma^{1/n}$ is concave, and we are done. 

We next state the real variable analog of Theorem 2.3.
\begin{thm}
Let $\phi$ be a convex function in the ball, continuous on the closed ball and vanishing on the boundary. Let $\hat\phi$ be its Schwarz symmetrization. Then
$$
\E(\hat\phi)\leq\E(\phi).
$$
\end{thm}
This is proved in a way completely parallell to the complex case, so we shall not give the details. We define geodesics and subgeodesics in the space of convex functions as before. Then the real energy is concave along subgeodesics and affine along geodesics as before and the analog of the formula for the first order derivative also holds. We can therefore repeat the proof practically verbatim.

\subsection{ Other domains}

We have already seen in section 2 that in the complex case, the energy does not in general decrease under Schwarz symmetrization if we consider functions defined on domains different than the ball. In the real setting, the first counterexample to the same effect was given by Tso, \cite{Tso}. We shall first  discuss Tso's counterexample, and start by giving the example in a more general form. In the next theorem appears the {\it Mahler volume} of a convex set $\Omega$ containing the origin. It is defined as
$$
M(\Omega):=|\Omega||\Omega^\circ|,
$$
where $\Omega^\circ$ is the polar body of $\Omega$. In the sequel we will write $\E_\Omega$ for the energy of functions defined in $\Omega$.

\begin{thm}
Let $\Omega$ be a bounded convex domain in $\R^n$ containing the origin, and let $\mu_\Omega$ be the Minkowski functional of $\Omega$. Let $u$ be a convex function in $\Omega$ of the form $u(x)=f(\mu_\Omega(x))$, and let $\hat u$ be its Schwarz symmetrization. Then
$$
M(\Omega)^{-1}\E_\Omega(u)=M(B_\Omega)^{-1}\E_B(\hat u)
$$
(where $B_\Omega$ is the ball of the same volume as $\Omega$).
\end{thm}
Notice that we could as well have divided by just $|\Omega^\circ|$ instead of the Mahler volume, since the volumes of $\Omega$ and $B_\Omega$ are automatically equal, but the Mahler volume seems to simplify a little below. 
From the theorem we see that if 
$$
\E_B(\hat u)\leq\E_\Omega(u)
$$
it follows that we have an inequality for the Mahler volumes
$$
M(B)\leq M(\Omega).
$$
This inequality fails in a very strong way. Indeed, if we assume that $\Omega$ is also symmetric so that $-\Omega=\Omega$, then {\it Santal\`o's inequality}, \cite{Santalo},  says that the opposite is true
$$
M(B)\geq M(\Omega).
$$ 
(Tsos's counterexample is the case of Theorem 4.3 when $\Omega$ is a simplex.)
It therefore seems that in the real case it is natural to normalize the energy by dividing by the Mahler volume of the domain. Notice that this is a difference as compared to the complex setting, where Theorem 4.3 holds without normalization. The reason for this is that in the case of $\R^n$, the Minkowski functional of a convex domain $\Omega$ satisfies the equation
$$
MA(\mu_\Omega)= |\Omega^0|.
$$
On the other hand in $\C^n$ we have if $\Omega$ is a balanced domain that
$$
(dd^c\log\mu_\Omega)^n=(dd^c\log |z|)^n= 2^{-n}\delta_0,
$$
where $\delta_0$ is a point mass at the origin, and thus is independent of the domain.

The question then becomes if 
\be
M(\Omega)^{-1}\E_\Omega(u)\geq M(B)^{-1}\E_B(\hat u)
\ee
for any convex function $u$ on $\Omega$ that vanishes on the boundary. 

Just like in the complex case we define for a convex function $u$ defined on some  convex domain $L$, its $\Omega$-symmetrization $S_\Omega(u)$ as the unique function, equidistributed with $u$ which can be written
$$
S_\Omega(u)=f(\mu_\Omega).
$$
Note that the $\Omega$-symmetrization of $u$ is the same as the $\Omega'$-symmetrization if $\Omega$ and $\Omega'$ are homothetic. Moreover,
since 
$$
|\{S_\Omega(u)<0\}|=|\{u<0\}|
$$
 $S_\Omega(u)$ vanishes on the boundary of a multiple $s\Omega$ of $\Omega$, with $s$ chosen so that $s\Omega$ has the same volume as $L$, if $u$ vanishes on the boundary of $L$.
Notice that if $\Omega$ is a ball, centered at the origin,  $S_\Omega$  is just the Schwarz symmetrization.

\bigskip
 
In terms of $\Omega$-symmetrizations, Theorem 4.3 says that the normalized energy of all $\Omega$-symmetrizations coincide:
\be
M(\Omega)^{-1}\E_\Omega(S_\Omega(u))=M(\Omega')^{-1}\E_{\Omega'}(S_{\Omega'}(u)),
\ee
if $\Omega$ and $\Omega'$ are two convex domains.

The desired inequality (4.1) thus means that
\be
\E_\Omega(S_\Omega(u))\leq \E_\Omega(u)
\ee
for convex functions $u$ on $\Omega$ that vanish on the boundary. This would be the analog of Theorem 4.2 for general convex domains, and it is precisely the same question that we discussed in the complex case. Just like in the complex case we shall now see that this holds only for ellipsoids. Most of the argument is completely parallell to the complex case and will be omitted. Only the last part, involving K\"ahler-Einstein metrics has to be changed and we shall now describe how this is done. 

\bigskip

As in the complex case we see that if the symmetrization inequality holds, then $\mu_\Omega $, the Minkowski functional of $\Omega$ satsifies a condition of the form: There is a convex function of $\mu_\Omega$ such that
$$
u=f(\mu_\Omega)
$$
satisfies an equation
$$
MA(u)= F(u)
$$
for some function $F$. To see the meaning of this more explicitly we resort to the complex formalism. Define $u$ and $\mu_\Omega$ on $\mathbb C^n$ by putting
$u(z)=u(x)$ etc, i e by letting all functions involved be independent of the imaginary part of $z$. Then
$$
MA(u)d\lambda(z)= C(dd^c u)^n
$$
where $d\lambda$ is the standard volume form on $\mathbb C^n$. Since $MA(\mu_\Omega)=0$ outside the origin, it follows if $u=f(\mu_\Omega)$ that
$$
c'(dd^c u)^n=(f'(\mu))^{n-1}f''(\mu)d\mu\wedge d^c \mu\wedge (dd^c \mu)^{n-1}.
$$
Hence we see that
$$
d\mu\wedge d^c \mu\wedge (dd^c \mu)^{n-1}=G(\mu) d\lambda
$$
for $x\neq 0$, where we write $\mu$ instead of $\mu_\Omega$ since $\Omega$ is now fixed. Since $\mu$ is homogenous of degree 1, $d\mu$ is homogenous of degree zero, and $dd^c\mu$ is homogenous of degree -1. Therefore the left hand side is homogenous of degree $(n-1)$ so we can take $G(\mu)=\mu^{1-n}$. It also follows from this equation that {\it any} function $u=f(\mu)$, with $f$ convex and strictly increasing must satsify an equation
$$
MA(u)= F(u)
$$
for some function $F$. Take $u=\mu^2$. Then $F(u)$ must be homogenous of degree zero, so $F(u)$ is a constant. All in all,  $u=\mu^2$ is outside of the origin a smooth convex function that satisfies
$$
MA(u)= C.
$$
Moreover, the second derivatives of $u$ stay bounded near the origin, so $u$ solves the same Monge-Ampere equation on all of $\mathbb R^n$ in a generalized sense. We can then apply a celebrated theorem by J\"orgens, Calabi and Pogorelov, (see \cite{Pogorelov}) to conclude that $u$ is a quadratic form. We have thus proved the next theorem.
\begin{thm} Let $\Omega$ be a convex domain containing the origin. Assume that for any convex function in $\Omega$, $v$ that vanishes on the boundary the symmetrization inequality
$$
M(\Omega)^{-1}\E_\Omega(v)\geq M(B)^{-1}\E_B(\hat v)
$$
holds. Then $v$ is an ellipsoid.
\end{thm}

\bigskip

{\bf Remark:} We saw above that the condition on our domain is that $\mu=\mu_\Omega$ satisfies an equation
$$
d\mu\wedge d^c \mu\wedge (dd^c \mu)^{n-1}= C \mu^{1-n} d\lambda.
$$
One can show that this is equivalent to the condition that $\Omega$ is a stationary point for the Mahler functional
$$
M(\Omega)=|\Omega||\Omega^\circ|.
$$
Thus it follows from the J\"orgens- Calabi-Pogorelov theorem that any such stationary point is an ellipsoid. Notice that the two results are not equivalent though: in the case of the Mahler functional we know beforehand that our function $u=\mu^2$ grows quadratically at infinity, whereas the  
 J\"orgens- Calabi-Pogorelov theorem applies to any convex solution. At any rate, the analogy between the K\"ahler-Einstein condition in the complex case and the Mahler volume in the real case seems quite interesting.
\qed

We conclude with the proof of Theorem 4.3, which is proved more or less as in the complex case. Notice that the appearance of the factor $|\Omega^\circ|$ in the lemma is the main difference as compared to Proposition 2.8.
\begin{lma} Let $\Omega$ be a smoothly bounded convex domain containing the origin, with Minkowski functional $\mu_\Omega$. Let $u$ be a smooth convex function in $\Omega$ of the form $u(x)=f(\mu_\Omega(x))$, vanishing on the boundary so that $f(1)=0$. Then
$$
\sigma(s):=\int_{\mu_\Omega<s} MA(u)= f'(s)^n |\Omega^\circ|.
$$
\end{lma}

\begin{proof}
We may assume that $u$ is strictly convex. Then the map $x\mapsto \nabla u(x)$ is a diffeomorphism from $\{\mu_\Omega\leq s\}$ to a domain $U_s$ in $\R^n$, and by the change of variables formula
$$
\sigma(s)= |U_s|.
$$
But $U_s$ only depends on the gradient map restricted to the boundary of the set $\Omega_s$ where $\mu_\Omega <s$, i e on the value of $f'(s)$. We may therefore take $f(s)=as$  and even, by homogenuity, take $a=1$. Then the boundary of $\Omega_s$ is mapped to the boundary of $\Omega^\circ$, so the volume is $|\Omega_\circ|$. 
\end{proof}
\begin{lma}Under the same hypotheses as in the previous lemma,
$$
\E(u)=\int^1 f'(s)^{n+1}ds |\Omega^\circ|.
$$
\end{lma}
\begin{proof}
We have
$$
\E(u)=-\int^1 f(s)d\sigma(s)=\int^1 f'(s)\sigma(s)ds,
$$
so this follows from the previous lemma.
\end{proof}
\begin{lma} Let $\Omega$ and $u$ be as in the previous lemmas, and let $B$ be a ball centered at the origin  of the same volume as $\Omega$. Then
$$
S_B(u)=f(\mu_B).
$$
\end{lma}
\begin{proof}
By definition, $S_B(u)=g(\mu_B)$ and
$$
|\{g(\mu_B)<t\}|=|\{f(\mu_\Omega\})<t\}|.
$$
The left hand side here is 
$$
|B| (g^{-1}(t))^n
$$
and the right hand side is 
$$
|\Omega| (f^{-1}(t))^n.
$$
Since $|B|=|\Omega|$, $g^{-1}=f^{-1}$, so we are done.
\end{proof}
Combining Lemma 4.6 and Lemma 4.5 we see that
$$
\E(u)/|\Omega^\circ|=\E(S_B)/|B^\circ|.
$$
This proves Theorem 4.3.
\def\listing#1#2#3{{\sc #1}:\ {\it #2}, \ #3.}


\begin{thebibliography}{9999}


 

\bibitem{Ber-Ber}\listing{Berman, R and Berndtsson, B}{Moser-Trudinger type inequalities for complex Monge-Amp\`ere operators and Aubin's 'hypoth\`ese fondamentale' }{arXiv:1109.1263}

\bibitem{Berg}\listing{Berndtsson, B}{Subharmonicity properties of the Bergman kernel and some other functions associated to pseudoconvex domains.}{Ann Inst Fourier 56 (2006) pp 1633-1662}


\bibitem{Chen}\listing{Chen X X}{The space of Kahler metrics}{J Differential Geometry, 56 (2000) pp 189-234} 

\bibitem{Kesavan}\listing{Kesavan, S}{Symmetrization and applications}{World Scientific, 2006, ISBN 981-256-733-X}

\bibitem{Moser}\listing{Moser, J}{A sharp form of an inequality by N. Trudinger.}{  Indiana Univ Math J (20) (1971), pp 1077-1092.}

\bibitem{Pogorelov}\listing{Pogorelov, A V}{On the improper convex affine hyperspheres}{Geometriae Dedicata 1 (1972), no. 1, 33-46.} 


\bibitem{Santalo}\listing{Santal\`o, L A}{An affine invariant for convex bodies of n-dimensional space.}{Port. Math 8 (1949), pp 155-161 }



\bibitem{Tso}\listing{Tso, Kaising}{ On symmetrization and Hessian equations.}{ J Analyse Math 52 (1989), pp 94-106.}







\end{thebibliography}
\end{document}